\documentclass{amsart}
\usepackage{amsmath, amsthm, amsfonts}
\newtheorem{lemma}{Lemma}[section]
\newtheorem{thm}{Theorem}[section]
\newtheorem{prop}{Proposition}[section]

\usepackage[cp1250]{inputenc}
\usepackage{amssymb}
\usepackage[T1]{fontenc}

\usepackage{graphicx}
\usepackage{amssymb}
\usepackage{times}
\usepackage{csquotes}
\usepackage{yfonts}
\usepackage{bbm}
\usepackage{dsfont}
\usepackage{enumitem}
\usepackage{bbm}
\usepackage{xcolor}

\title{Doob's estimate for coherent random variables and maximal operators on trees}

\author{Stanis\l{}aw Cichomski}
\address{Department of Mathematics, Informatics and Mechanics\\
 University of Warsaw\\
Banacha 2, 02-097 Warsaw\\
Poland}

\author{Adam Os\k{e}kowski}
\address{Department of Mathematics, Informatics and Mechanics\\
 University of Warsaw\\
Banacha 2, 02-097 Warsaw\\
Poland}

\numberwithin{equation}{section}

\begin{document}

\begin{abstract} Let $\xi$ be an integrable random variable defined on  $(\Omega, \mathcal{F}, \mathbb{P})$. Fix $k\in \mathbb{Z}_+$ and let $\{\mathcal{G}_{i}^{j}\}_{1\le i \le n,  1\le j \le k}$ be a reference family of  sub-$\sigma$-fields of $\mathcal{F}$, such that  $\{\mathcal{G}_{i}^{j}\}_{1\le i \le n}$ is a filtration for each $j\in \{1,2,\dots,k\}$.
  In this article we explain the underlying connection between the 
analysis of  the maximal functions of the corresponding coherent vector and  basic combinatorial properties of the uncentered Hardy--Littlewood maximal operator. Following a classical approach of Grafakos, Kinnunen and Montgomery-Smith, we establish an appropriate  version of the celebrated Doob's  maximal estimate. \end{abstract}

\maketitle

\section{Introduction}
The inspiration for the results obtained in this paper comes from the recent developments in the theory of coherent distributions. To introduce the necessary notions, suppose that $(\Omega, \mathcal{F}, \mathbb{P})$ is an arbitrary nonatomic probability space. Following \cite{pitman}, we say that  a random vector $X=(X_1,X_2,\ldots,X_n)$ is coherent, if there exist a 
random variable $\xi$ taking values in $[0,1]$ and a sequence $\mathcal{G}=(\mathcal{G}_1$, $\mathcal{G}_2$, $\ldots$, $\mathcal{G}_n)$ of sub-$\sigma$-algebras of $\mathcal F$ such that $X_k=\mathbb E(\xi|\mathcal{G}_k)$ for all $k=1,\,2,\,\ldots,\,n$. The motivation for this definition lies from economics, where coherent distributions are used to model the behavior of agents with partially overlapping information sources \cite{post,PPI}. From the mathematical point of view, such random vectors enjoy many interesting structural properties; for some latest theoretical advances on this subject, see e.g. \cite{contra, EJP675, BPC}. In this article, we will be interested in the universal sharp norm comparison of $\xi$ and the maximal function of $X$. We will drop the assumption $\mathbb{P}(\xi\in [0,1])=1$ and work with arbitrary integrable random variables. For such a $\xi$ and a sequence $\mathcal{G}$, the associated maximal function is given by $M_\mathcal{G}\xi=\sup_j |\mathbb{E}(\xi|\mathcal{G}_j)|.$ 
The starting point is the classical result of Doob, which asserts that
\begin{equation}\label{Doob}
 \Big\|M_\mathcal{G}\xi\Big\|_p\leq \frac{p}{p-1}\|\xi\|_p,\qquad 1<p\leq \infty,
\end{equation}
in the case when $\mathcal{G}$ is a filtration, i.e., we have the nesting condition $\mathcal{G}_1\subseteq \mathcal{G}_2\subseteq \ldots \subseteq \mathcal{G}_n$. Furthermore, for each $p$ the number $p/(p-1)$ is the best universal constant (i.e., not depending on the length of $\mathcal{G}$) allowed in the estimate. The main goal of this paper is to consider \eqref{Doob} for more general families of $\sigma$-algebras: we will assume that $\mathcal{G}$ can be decomposed into the union of filtrations. Specifically, we let $\mathcal{G}$ be of the form
$$ \mathcal{G}:=\Big\{\mathcal{G}_{i}^{j}\Big\}_{\substack{1\le i \le n, \\ 1\le j \le k}},$$
and require  the inclusions $\mathcal{G}_1^j\subseteq \mathcal{G}_2^j\subseteq \ldots \subseteq \mathcal{G}_n^j$ for each $j$. No relation between $\sigma$-algebras $\mathcal{G}^j_i$ with different $j$ is imposed.  Thus, our investigation can be seen as a natural halfway state between the study of general coherent distributions and classical martingales.
 Furthermore, this subject enters into the still vague framework of martingales indexed by partially ordered sets. For a general introduction to this theory see \cite{SIM}, for related Doob's type inequalities see \cite{plane2, array, tree, plane1}. Our reasoning will reveal an unexpected connection between the  analysis of $\max_{i,j}|\mathbb{E}(\xi|\mathcal{G}_{i}^j)|$ and basic combinatorial properties of the uncentered Hardy--Littlewood maximal operator on tree-shaped domains. Due to this interdependence, we will be able to extend the  classical approach introduced in \cite{grafakos2, grafakos1}  and derive an appropriate sharp version of \eqref{Doob}.

\begin{thm}\label{Db}
Let $1<p<\infty$ be a given parameter and let $\mathcal{G}=\big\{\mathcal{G}_{i}^{j}\big\}_{\substack{1\le i \le n, 1\le j \le k}}$ be the union of filtrations as above. Then for any random variable $\xi\in L^p$ we have the estimate
\begin{equation}\label{Doobm}
 \|M_\mathcal{G}\xi\|_p\leq C_{p,k}\|\xi\|_p,
\end{equation}
where $C_{p,k}$ is the unique root of the equation
\begin{equation}\label{defC}
 (p-1)C_{p,k}^p-pC_{p,k}^{p-1}-(k-1)=0.
\end{equation}
For fixed $1<p<\infty$ and $k\geq 1$, the constant $C_{p,k}$ is the best possible: given $\varepsilon>0$, there is an integer $n$, a family $\mathcal{G}$ as above and a random variable $\xi\in L^p$ for which
 $$  \|M_\mathcal{G}\xi\|_p>( C_{p,k}-\varepsilon)\|\xi\|_p.$$
\end{thm}
That is, the constant $C_{p,k}$ is the best universal constant allowed in \eqref{Doobm}, where the universality is the non-dependence on $n$, the length of the filtrations building $\mathcal{G}$.

\smallskip

We turn our attention to the analytic contents of the paper. Let $k$ be a fixed positive integer. Consider the set $\mathcal{R}_k=\bigcup_{j=1}^k H_j$, where $H_j$ is the line segment on the complex plane,  with endpoints $0$ and $e^{2\pi i j/k}$, $j=1,\,2,\,\ldots,\,k$. That is, $\mathcal{R}_k$ is a tree-shaped domain being the union of $k$  rays $H_1$, $H_2$, $\ldots$, $H_k$, each having length one. We equip $\mathcal{R}_k$ with the standard British railway metric and the normalized one-dimensional Lebesgue measure $\lambda_k$. Then we can introduce the concept of the decreasing rearrangement on $\mathcal{R}_k$. Namely, for an arbitrary random variable $\xi$ on $(\Omega,\mathcal{F},\mathbb{P})$, we define first its distribution function $d_\xi:[0,\infty)\to [0,1]$ by $ d_\xi(s)=\mathbb{P}(|\xi|>s)$. Then the associated $k$-decreasing rearrangement $\xi^*_{(k)}:\mathcal{R}_k\to [0,\infty)$ is given by
$$ \xi^*_{(k)}(e^{2\pi ij/k}t)=\inf\{s>0\,:\,d_\xi(s)\leq t\},\qquad j=1,\,2,\,\ldots,\,k.$$
Equivalently, $\xi^*_{(k)}$ can be defined by taking the standard decreasing rearrangement $\xi^*$ on $[0,1]$ and copying it on each ray $H_j$, in accordance with the natural order induced by the distance from $0$. Thus, we immediately see that  $|\xi|$ and $\xi^*_{(k)}$ have the same distributions (as random variables on $\Omega$ and $\mathcal{R}_k$, respectively). Furthermore, $\xi^*_{(k)}$ is radially decreasing, i.e., $\xi^*_{(k)}(x)=\xi^*_{(k)}(|x|)$ decreases as $|x|$ grows.%; it is also easy to check that $\xi^*_{(k)}$ is continuous from the right along the rays.

Finally, we introduce  the uncentered Hardy-Littlewood maximal function $\mathcal{M}_{(k)}$ in the above setup. This operator acts on integrable functions $f$ on $\mathcal{R}_k$ by the usual formula
$$ \mathcal{M}_{(k)}f(x)=\sup \frac{1}{\lambda_k(B)}\int_B |f|\mbox{d}\lambda_k,\qquad x\in \mathcal{R}_k,$$
where the supremum is taken over all open balls $B\subseteq \mathcal{R}_k$ which contain $x$. We will identify the $L^p$ norm of this object.

\begin{thm}\label{HL}
For any $1<p<\infty$ and any $k\geq 1$ we have $\|\mathcal{M}_{(k)}\|_{L^p\to L^p}=C_{p,k}$, where $C_{p,k}$ is given in \eqref{defC}.
\end{thm}

For $k=1$, this is the classical result of Hardy, the case $k=2$ was established by Grafakos and Montgomery-Smith \cite{grafakos1}, our contribution is the analysis for $k\geq 3$. Furthermore, we will link the context of coherent distributions with the analytic setup above, intertwining the contents of Theorems \ref{Db} and \ref{HL}.

\begin{thm} \label{mainL}  
Let $k,\,n\geq 1$ be fixed integers.  
Suppose further that $\xi$ is an integrable random variable and assume that $\mathcal{G}=\big\{\mathcal{G}_{i}^{j}\big\}_{\substack{1\le i \le n, 1\le j \le k}}$ is a union of filtrations as above. Then the maximal function  $M_\mathcal{G}\xi$ satisfies the majorization
\begin{equation} \label{mainLE} 
(M_\mathcal{G}\xi)^*_{(k)}\leq \mathcal{M}_{(k)} (\xi^*_{(k)})\qquad \lambda_k\mbox{-almost everywhere on } \mathcal{R}_k.
\end{equation}
\end{thm}

The remaining part of the paper is split into two sections. In Section 2 we establish  Theorem \ref{mainL}. In the last part of the paper, we establish the $L^p$ bound $\|\mathcal{M}^k\|_{L^p\to L^p}\leq C_{p,k}$, which allows us to deduce \eqref{Doobm} immediately. Furthermore, we show there the sharpness of the latter inequality, thus completing the proofs of all aforementioned results.

From now on, the parameter $k$ will be kept fixed; to simplify the notation, we will skip the index and write $\xi^*$, $\mathcal{M}$ instead of $\xi^*_{(k)}$ and $\mathcal{M}_{(k)}$, respectively.

 \section{Proof of Theorem \ref{mainL}}
 
 We will need the following property of the Hardy-Littlewood maximal operator.
 
 \begin{lemma}\label{aux}
 Suppose that $\xi$ is an integrable random variable. Then for any $s>0$ such that $\lambda_k(\mathcal M\xi^*>s)<1$ we have
\begin{align*}
 &s\Big((k-1)\lambda_k(\xi^*>s)+\lambda_k(\mathcal M\xi^*>s)\Big)\\
&\quad \qquad \qquad =(k-1)\int_{\{\xi^*>s\}}\xi^*\mbox{d}\lambda_k+\int_{\{\mathcal M\xi^*>s\}} \xi^*\mbox{d}\lambda_k.
 \end{align*}
 \end{lemma}
 \begin{proof}
If $s\geq \|\xi\|_\infty$, then the assertion is evident (both sides are zero), so from now on we assume that $s<\|\xi\|_\infty$. The  
 function $\mathcal M\xi_{(k)}^*$ is radially decreasing along the rays of $\mathcal{R}_k$. Furthermore, it is continuous, which follows directly from Lebesgue's dominated convergence theorem. Thus there exists $u\in \mathcal{R}_k$, lying on the ray $H_1$, for which $s=\mathcal M\xi^*(u)$. It is easy to identify the ball $B$ for which the supremum defining $\mathcal M\xi^*(u)$ is attained: $u$ must be one of its boundary points, and the intersection $B\cap H_j$ for $j\neq 1$ must be the part of $H_j$ on which we have $f>s$. It remains to note that the equality
$$ s=M\xi^*(u)=\frac{1}{\lambda_k(B)}\int_B \xi^*\mbox{d}\lambda_k$$
is equivalent to the claim. Indeed, we have $\lambda_k(B)=\frac{k-1}{k}\lambda_k(\xi^*>s)+\frac{1}{k}\lambda_k(\mathcal M\xi^*>s)$, with a similar identity for $\int_B \xi^*\mbox{d}\lambda_k$.
 \end{proof}

\begin{proof}[Proof of Theorem \ref{mainL}] It is enough to show the tail inequality
\begin{equation}\label{tail}
 \mathbb{P}(M_\mathcal{G}\xi>s)\leq \lambda_k(\mathcal{M}\xi^*>s)
\end{equation}
for all $s$. Now we consider two separate steps.

\smallskip

\emph{Step 1. Reductions.} Let us first exclude the trivial cases: from now on, we will assume that $\lambda_k(\mathcal M\xi^*>s)<1$ and $s<\|\xi\|_\infty$. Indeed, if $\lambda_k(\mathcal M\xi^*>s)=1$, then there is nothing to prove, while for $s\geq \|\xi\|_\infty$ both sides of \eqref{tail} are zero. 
Adding the full $\sigma$-algebras $\mathcal{G}^j_{n+1}=\mathcal F$, $j=1,\,2,\,\ldots,\,k$ to the collection $\mathcal{G}$ if necessary, we may and do assume that 
\begin{equation}\label{light}
\max_i |\mathbb E(\xi|\mathcal{G}^j_i)|\geq |\xi|\qquad \mbox{ almost surely for all $j$.}
\end{equation}
In particular, this gives $M_\mathcal{G}\xi\geq |\xi|$ with probability $1$. %Our final reduction is that we postulate the existence of $u$ such that
%\begin{equation}\label{monot}
% \{|\xi|> u\}\subseteq \{M_\mathcal{G}\xi>s\}\subseteq \{|\xi|\geq u\}.
%\end{equation}
%To see this, note that for $u=\inf\{t:\mathbb{P}(|\xi|\geq t)\leq \mathbb{P}(S_\mathcal{G}>s)\}$ we have $\mathbb{P}(|\xi|>u)\leq \mathbb{P}(S_\mathcal{G}>s)\leq \mathbb{P}(|\xi|\geq u)$. Now, if \eqref{monot} does not hold, we replace $\xi$ appropriately. Namely, consider the (problematic) event $A=\{S_\mathcal{G}>s,\,|\xi|<u\}$ and note that there is a measurable subset $B$ of $\{S_\mathcal{G}\leq s,\,|\xi|\geq u\}$ of the same probability. This follows directly from the inequality $\mathbb{P}(S_\mathcal{G}>s)\leq \mathbb{P}(|\xi|\geq u)$. Now we switch the two events: we replace $\xi$ with a random variable $\tilde{\xi}$ with $\xi\equiv\tilde{\xi}$ outside $A\cup B$ and such that the conditional distributions of $\xi$ on $A$ and $\tilde{\xi}$ on $B$ coincide. Then $\xi$ and $\tilde{\xi}$ have the same distribution, so the right-hand side of \eqref{tail} does not increase; on the other hand, we have increased the variable on the set $\{S_\mathcal{G}>s\}$, and hence the left-hand side can only increase.

\smallskip

\emph{Step 2. Proof of theorem.} Fix an arbitrary  $s>0$ and write
 $$ \mathbb{P}(M_\mathcal{G}\xi>s)=\mathbb{P}(A_1\cup A_2\cup \ldots \cup A_k),$$
 where $A_j=\{\max_i |\mathbb{E}(\xi|\mathcal{G}^j_i)|>s\}$, $j=1,\,2,\,\ldots,\,k$. Let us distinguish the additional event $A_0=\{|\xi|>s\}$ and observe that $A_0\subseteq A_j$ for each $j$, in the light of \eqref{light}. Note that if $\tilde{A}_j$ is an arbitrary event satisfying $A_0\subseteq \tilde{A}_j\subseteq A_j$, then we have
 \begin{equation}\label{weaktype0}
 s\mathbb{P}(\tilde{A}_j)- \int_{\tilde{A}_j}|\xi|\mbox{d}\mathbb{P}=\int_{\tilde{A}_j}(s-|\xi|)\mbox{d}\mathbb{P}\leq \int_{A_j}(s-|\xi|)\mbox{d}\mathbb{P}\leq 0,
 \end{equation}
where the latter bound follows from Doob's weak-type bound for martingale maximal function. 
 Next, we write
 \begin{align*}
 &\mathbb{P}(A_1\cup A_2\cup \ldots \cup A_k)\\
&=\mathbb{P}(A_0\cup A_1\cup A_2\cup \ldots \cup A_k)\\
 &=\mathbb{P}(A_0)+\mathbb{P}(A_1\setminus A_0)+\mathbb{P}(A_2\setminus (A_1\cup A_0))+\ldots+\mathbb{P}(A_n\setminus (A_{n-1}\cup A_{n-2}\cup \ldots A_0)).
 \end{align*}
 Set $\tilde{A}_j=A_0\cup \big(A_j\setminus (A_{j-1}\cup A_{j-2}\cup \ldots \cup A_0)\big)$, apply \eqref{weaktype0} and add the  estimates over $j$. Combining the result with the above formula for $\mathbb{P}(A_1\cup A_2\cup \ldots \cup A_k)$, we obtain
 \begin{align*}
 s\Big[\mathbb{P}(A_1\cup A_2\cup \ldots \cup A_k)+(k-1)\mathbb{P}(A_0)\Big]&=s\sum_{j=1}^k \mathbb{P}(\tilde{A}_j)\leq \sum_{j=1}^k \int_{\tilde{A}_j} |\xi|\mbox{d}\mathbb{P},
\end{align*} 
or equivalently,
$$ s\big[\mathbb{P}(M_\mathcal{G}\xi>s)+(k-1)\mathbb{P}(A_0)\big]\leq \int_{\{M_\mathcal{G}\xi>s\}}|\xi|\mbox{d}\mathbb{P}+(k-1)\int_{A_0}|\xi|\mbox{d}\mathbb{P}.$$
Since $|\xi|$ and $\xi^*$ are equidistributed, we have $\mathbb{P}(A_0)=\lambda_k(\xi^*>s)$ and $\int_{A_0}|\xi|\mbox{d}\mathbb{P}=\int_{\{\xi^*>s\}}\xi^*\mbox{d}\lambda_k$. Plugging this above and applying Lemma \ref{aux}, we get
$$ \int_{\{M_\mathcal{G}\xi>s\}}(s-|\xi|)\mbox{d}\mathbb{P}\leq \int_{\{M\xi^*>s\}} (s-\xi^*)\mbox{d}\lambda_k,$$
or, subtracting the equality $\int_{\{|\xi|>s\}}(s-|\xi|)\mbox{d}\mathbb{P}=\int_{\{\xi^*>s\}}(s-\xi^*)\mbox{d}\lambda_k$,
$$ \int_{\{M_\mathcal{G}\xi>s\}}(s-|\xi|)_+\mbox{d}\mathbb{P}\leq \int_{\{\mathcal M\xi^*>s\}} (s-\xi^*)_+\mbox{d}\lambda_k=\int_{\mathcal{R}_k}\chi_{\{\mathcal M\xi^*>s\}} (s-\xi^*)_+\mbox{d}\lambda_k.$$
However, the nonnegative functions $\chi_{\{\mathcal M\xi^*>s\}}$ and $(s-\xi^*)_+$ have the reversed monotonicity along the rays: the first of them is non-increasing, while the second is non-decreasing. Since $(s-|\xi|)_+$ and $(s-\xi^*)_+$ have the same distribution, \eqref{tail} follows.
\end{proof}

 \section{$L^p$ estimates}

We turn our attention to Theorems \ref{Db} and \ref{HL}. First we will establish the $L^p$ bound for the uncentered maximal operator; the key ingredient of the proof is the following weak-type estimate.

\begin{prop}
For an arbitrary integrable function $f$ on $\mathcal{R}_k$ and any $s>0$ we have
\begin{equation}\label{weaktype}
 s\lambda_k(\mathcal{M}f>s)+s(k-1)\lambda_k(|f|>s)\leq \int_{\{\mathcal{M}f>s\}}|f|\mbox{d}\lambda_k+(k-1)\int_{\{|f|>s\}}|f|\mbox{d}\lambda_k.
\end{equation}
\end{prop}
\begin{proof}
It is convenient to split the reasoning into two steps.

\smallskip

\emph{Step 1. Special balls in $\mathcal{R}_k$}. Let us consider the level set $E=\{x\in \mathcal{R}_k\,:\,\mathcal{M}f>s\}$. %Fix $j\in \{1,2,\,\ldots,\,k\}$ and let $E_j=\{x\in H_j\,:\,\mathcal{M}f>s\}$. 
Then for each $x\in E$ there is an open ball $B_x\subseteq\mathcal{R}_k$ which contains $x$ and satisfies $ \lambda_k(B_x)^{-1}\int_{B_x} |f|\mbox{d}\lambda_k>s$. This inequality implies that $B_x\subseteq E$ and hence $\bigcup_{x\in E}B_x=E$. By the Lindel\H of's theorem, we may pick a countable subcollection $(B_{x_n})_{n=1}^\infty$ such that $\bigcup_{n=1}^\infty B_{x_n}=E$. With no loss of generality, we may assume that $B_{x_i}$ is not a subset of $B_{x_j}$ for $i\not=j$. We fix an integer $N$ and restrict ourselves to the finite family $\mathcal{B}=(B_{x_n})_{n=1}^N$. The idea is to pick a subcollection $\mathcal{B}'$ of $\mathcal{B}$ which does not overlap too much. To this end, we will choose appropriate balls from each separate ray of $\mathcal{R}_k$, exploiting the natural order induced by the distance from $0$. For simplicity, we will only describe the procedure for the $k$-th ray (i.e., for the interval $[0,1]$), the argument for other rays is the same, up to rotation.

First, we pick a ball from $\mathcal{B}$ which contains zero and call it $J_0$ (if no ball in $\mathcal{B}$ contains zero, we let $J_0=\emptyset$; if there are several balls with this property, we take the ball whose intersection with $[0,1]$ has the biggest measure). Next we apply the following inductive procedure.

\smallskip

1$^\circ$ Suppose that we have successfully defined $J_{n}$. Consider the family 
of all intervals $J\in \mathcal{B}$ which intersect $J_{n}$ and satisfy $\sup J > \sup J_{n}$. If 
this family is nonempty, choose the interval with largest left-endpoint (if this object is not unique, pick the one with the biggest measure) and 
denote it by $J_{n+1}$.

2$^\circ$ If the family in 1$^\circ$ is empty, then consider all intervals $J\in \mathcal{B}$ with
$\inf J \ge \sup J_{n}$. If this family is nonempty, choose an element with the 
smallest left-endpoint (again, if this object is not unique, pick the one 
with the biggest measure) and denote it by $J_{n+1}$.

3$^\circ$ Go to 1$^\circ$.

\smallskip

Since the family $\mathcal{B}$ is finite, the procedure stops after a number of sets (in 1$^\circ$ and 2$^\circ$, there are no balls to choose from) and returns a family $J_0^j$, $J_1^j$, $J_2^j$, $\ldots$, $J_{m_j}^j$ of balls. Observe that by the very construction, $J_0^j$, $J_2^j$, $J_4^j$, $\ldots$ are pairwise disjoint and the same is true for $J_1^j$, $J_3^j$, $J_5^j$,  $\ldots$. Letting 
$$\mathcal{B}'=\Big\{J_\ell^j\,:\, 1\leq \ell\leq m_j,\,j=1,\,2,\,\ldots,\,k\Big\},$$
we easily check that 
\begin{equation}\label{union}
\bigcup_{B\in \mathcal{B}}B=\bigcup_{B\in \mathcal{B}'}B.
\end{equation} 
Next, by the disjointness properties of the sequences $J^j_i$, note that a family $\mathcal{B}'$ has the following property: each point $x\in \mathcal{R}_k$ belongs to at most $k+1$ elements of $\mathcal{B}'$. Moreover, we can actually improve this last bound by $1$. Now, say that there is a point $x_0\in \mathcal{R}_k$ which belongs to exactly $k+1$ elements of $\mathcal{B}'$ and let us assume that $x_0$ belongs to the the $k$-th ray $H_k$. By the extremality of $J_0^k$ we must have $(J_0^i \cap [0,1])  \subset (J_0^k \cap [0,1])$ for all $i=1,2,\dots,k-1$, and hence
$$x_0 \in \bigcap_{j=1}^k J_0^j \cap J_1^k.$$

Thus, we simply remove $J_0^k$ from the family $\mathcal{B}'$. Such a modification does not affect the validity of \eqref{union} and proves our assertion.

\smallskip

\emph{Step 2. Calculation.} Since $\mathcal{B}'\subseteq \mathcal{B}$, each element $B$ of $\mathcal{B}'$ satisfies
$$ s\lambda_k(B)\leq \int_B |f|\mbox{d}\lambda_k.$$
Summing over all $B\in \mathcal{B}'$, we thus obtain
$$ s\left[\lambda\left(\bigcup_{B\in \mathcal{B}'}B\right)+\sum_{j=2}^k \lambda_k(A_j)\right]\leq \int_{\bigcup_{B\in \mathcal{B}'}B}|f|\mbox{d}\lambda_k+\sum_{j=2}^k \int_{A_j}|f|\mbox{d}\lambda_k,$$
where $A_j$ is the collection of all $x\in \mathcal{R}_k$ which belong to exactly $j$ elements of $\mathcal{B}'$. This is equivalent to
\begin{align*}
s\lambda\left(\bigcup_{B\in \mathcal{B}}B\right)&\leq \int_{\bigcup_{B\in \mathcal{B}}B}|f|\mbox{d}\lambda_k+\sum_{j=2}^k \int_{A_j}(|f|-s)\mbox{d}\lambda_k\\
&\leq \int_{\bigcup_{B\in \mathcal{B}}B}|f|\mbox{d}\lambda_k+\sum_{j=2}^k \int_{A_j}(|f|-s)_+\mbox{d}\lambda_k\\
&\leq \int_{\bigcup_{B\in \mathcal{B}}B}|f|\mbox{d}\lambda_k+(k-1)\int_{\bigcup_{j=2}^k A_j}(|f|-s)_+\mbox{d}\lambda_k\\
&\leq \int_{\bigcup_{B\in \mathcal{B}}B}|f|\mbox{d}\lambda_k+(k-1)\int_{\mathcal{R}_k}(|f|-s)_+\mbox{d}\lambda_k.
\end{align*}
Now recall that the family $\mathcal{B}$ depended on $N$. Letting this parameter to infinity and using Lebesgue's monotone convergence theorem, we obtain
$$ s\lambda(E)\leq \int_E |f|\mbox{d}\lambda_k+(k-1)\int_{\mathcal{R}_k}(|f|-s)_+\mbox{d}\lambda_k.$$
This is precisely the claim.
\end{proof}

Now, using the standard integration argument, we obtain the $L^p$ estimate for the uncentered maximal operator on $\mathcal{R}_k$.

\begin{proof}[Proof of \eqref{Doobm}] By Fubini's theorem, we have
\begin{align*}
 &\int_{\mathcal{R}_k} \big(\mathcal{M} f\big)^p\mbox{d}\lambda_k+(k-1)\int_{\mathcal{R}_k}|f|^p\mbox{d}\lambda_k\\
 &=p\int_0^\infty s^{p-1}\left[\lambda_k(\mathcal{M} f>s)+(k-1)\lambda_k(|f|>s)\right]\mbox{d}s,
 \end{align*}
 which by \eqref{weaktype} does not exceed
 \begin{align*}
 &p\int_0^\infty s^{p-2}\left[\int_{\{\mathcal{M}f>s\}}|f|\mbox{d}\lambda_k+(k-1)\int_{\{|f|>s\}}|f|\mbox{d}\lambda_k\right]\mbox{d}s\\
&= \frac{p}{p-1} \int_{\mathcal{R}_k} \Big((\mathcal{M}f)^{p-1}|f|+(k-1)|f|^p\Big)\mbox{d}\lambda_k.
\end{align*}
Here in the last passage we have used Fubini's theorem again. This gives the bound
$$ \int_{\mathcal{R}_k} \big(\mathcal{M} f\big)^p\mbox{d}\lambda_k\leq \frac{p}{p-1}\int_{\mathcal{R}_k} \big(\mathcal{M}f\big)^{p-1}|f|\mbox{d}\lambda_k+\frac{k-1}{p-1}\int_{\mathcal{R}_k}|f|^p\mbox{d}\lambda_k.$$
However, by H\"older's inequality, we have
$$ \int_{\mathcal{R}_k} \big(\mathcal{M}f\big)^{p-1}|f|\mbox{d}\lambda_k\leq \left(\int_{\mathcal{R}_k}\big(\mathcal{M}f\big)^p\mbox{d}\lambda_k\right)^{(p-1)/p}\left(\int_{\mathcal{R}_k}|f|^p\mbox{d}\lambda_k\right)^{1/p},$$
which combined with the previous estimate yields
 $$(p-1)\Bigg( \frac{\| \mathcal{M}f\|_{L^p(\mathcal{R}_k)}}{\|f\|_{L^p(\mathcal{R}_k)}}\Bigg)^p  -  p\Bigg(\frac{\| \mathcal{M}f\|_{L^p(\mathcal{R}_k)}}{\|f\|_{L^p(\mathcal{R}_k)}}\Bigg)^{p-1} \ - \ (k-1) \le  0.$$
It remains to note that the function $s\mapsto (p-1)s^p-ps^{p-1}-(k-1)$ is increasing on $[1,\infty)$ and $C_{p,k}$ is its unique root. 
This establishes the desired $L^p$ bound $\|\mathcal{M}f\|_{L^p(\mathcal{R}_k)}\leq C_{p,k}\|f\|_{L^p(\mathcal{R}_k)}$.
\end{proof}

Combining the $L^p$ estimate we have just proved with the inequality \eqref{mainLE}, we immediately obtain \eqref{Doobm}, Doob's inequality for the coherent random variables. It remains to prove the optimality of the constant $C_{p,k}$ in the latter estimate. Having proved this sharpness, we immediately deduce the optimality of the constant for the uncentered maximal operator.

\begin{proof}[Sharpness of $C_{p,k}$ for coherent variables] Let $1<p<\infty$ and $k\in \{1,2,\ldots\}$ be fixed.  Consider the probability space $\mathcal{R}_k$ with its Borel subsets and normalized one-dimensional Lebesgue's measure $\lambda_k$. Fix an auxiliary constant $r\in (0,p^{-1})$ and consider the random variable $\xi(\omega)=|\omega|^{-r}$: then the estimate $r<p^{-1}$ guarantees that this variable belongs to $L^p$. To define the filtrations, let $\lambda_{r,k}$ be the unique root of the equation
\begin{equation}\label{lam}
\lambda(1-r)-(k-1)r\lambda^{(r-1)/r}-1=0,\qquad  1\leq \lambda<\infty.
\end{equation}
The existence and uniqueness of $\lambda_{r,k}$ is direct consequence of the fact that the left-hand side, considered as a function of $\lambda$, is strictly increasing, negative at $\lambda=1$ and positive for large $\lambda$. Now, for any $j\in \{1,2,\ldots,k\}$, introduce the closed ball $B_j$ which has the center $e^{2\pi i j/k} (1-\lambda_{r,k}^{-1/r})/2$ and radius $(1+\lambda_{r,k}^{-1/r})/2.$ This ball covers the whole ray $H_j$ and some portion of the remaining rays: $|B_j\cap H_i|=\lambda_{r,k}^{-1/r}$ for $i\neq j$. Therefore if $x$ lies on the $j$-th ray of $\mathcal{R}_k$, then the rescaled ball $|x|B_j=\{|x|y\in \mathcal{R}_k:y\in B_j\}$ satisfies
$$
 \frac{1}{\lambda_k(|x|B_j)}\int_{|x|B_j} \xi \mbox{d}\lambda_k=\frac{\int_0^{|x|}\omega^{-r}\mbox{d}\omega+(k-1)\int_0^{\lambda_{r,k}|x|} \omega^{-r}\mbox{d}\omega}{|x|+\lambda_{r,k}|x|}=\lambda_{r,k}\cdot \xi(x),
$$
by \eqref{lam}. Since both sides are homogeneous of order $-r$ (as a function of $x$), one can actually show a bit more: for any $\varepsilon >0$ there is $\delta\in (0,1)$ such that if $y\in H_j$ satisfies $\delta<|y/x|\leq 1$, then
\begin{equation}\label{aver}
 \frac{1}{\lambda_k(|x|B_j)}\int_{|x|B_j} \xi \mbox{d}\lambda_k\geq (\lambda_{r,k}-\varepsilon)\cdot \xi(y).
\end{equation}
Fix $\varepsilon,\,\delta$ with the above property and pick a large integer $N$. For any $n=0,\,1,\,2,\,\ldots,\,N$, let $\mathcal{G}^j_n$ be the $\sigma$-algebra generated by the balls $B_j$, $\delta B_j$, $\delta^2B_j$, $\ldots$, $\delta^{n-1} B_j$. It follows directly from \eqref{aver} that 
$$ M_\mathcal{G}\xi \geq (\lambda_{r,k}-\varepsilon )\xi\qquad \mbox{almost surely on }\mathcal{R}_k\setminus \delta^N B_j.$$
But $\xi\in L^p$, as we have already discussed above. Since $\varepsilon$ and $N$ were taken arbitrarily, the best constant allowed in the estimate \eqref{Doobm} is at least $\lambda_{r,k}$. It remains to note that if we let $r\to p^{-1}$, then $\lambda_{r,k}$ converges to the constant $C_{p,k}$: in the limit, the equation \eqref{lam} becomes \eqref{defC}. This proves the desired sharpness.
\end{proof}

%-----------------------------------------------------------------

\setlength{\baselineskip}{2ex}

\bibliographystyle{plain}
\bibliography{kFILbib}

\end{document}